\documentclass[11pt,leqno]{amsart}

\newtheorem{theorem}{\sc Theorem}[section]
\newtheorem{lemma}[theorem]{\sc Lemma}
\newtheorem{proposition}[theorem]{\sc Proposition}

\newtheorem{hypothesis}[theorem]{\sc Hypothesis}

\begin{document}

\author{Pavel Shumyatsky}
\address{Department of Mathematics, University of Brasilia,
Brasilia-DF, 70910-900 Brazil ,pavel@unb.br}
\email{pavel@mat.unb.br}

\title{On the Exponent of a Verbal Subgroup in a Finite Group}

\begin{abstract}
Let $w$ be a multilinear commutator word. We prove that if $e$ is a positive integer and $G$ is a finite group in which any nilpotent subgroup generated by $w$-values has exponent dividing $e$ then the exponent of the corresponding verbal subgroup $w(G)$ is bounded in terms of $e$ and $w$ only.
\end{abstract}

\maketitle

\section{Introduction}

A number of outstanding results about words in finite groups have been obtained in recent years. In this context we mention Shalev's theorem that for any non-trivial group word $w$, every element of every sufficiently large finite simple group is a product of at most three $w$-values \cite{shal}, and the proof by Liebeck, O'Brien, Shalev and Tiep \cite{lost} of Ore's conjecture: Every element of a finite simple group is a commutator. Another significant result is that of Nikolov and Segal that if $G$ is an $m$-generated finite group, then every element of $G'$ is a product of $m$-boundedly many commutators \cite{nisegal}. 

Our interest in words began in \cite{ams99} where it was shown that if $G$ is a residually finite group in which all commutators have orders dividing a given prime-power, then the derived group $G'$ is locally finite. Later, in \cite{lola,vari,comu68,deltacom} we treated other problems on local finiteness of verbal subgroups in residually finite groups. Inevitably, at a crucial point we had to deal with questions about the exponent of a verbal subgroup of a finite group.   

Recall that a group has exponent $e$ if $x^e=1$ for all $x\in G$ and $e$ is the least positive integer with that property. Given a word $w$, we denote by $w(G)$ the verbal subgroup of $G$ generated by the values of $w$. The goal of the present paper is to prove the following theorem.
\medskip

{Theorem A.} {\it Let $w$ be a multilinear commutator and $G$ a finite group in which any nilpotent subgroup generated by $w$-values has exponent dividing $e$. Then the exponent of the verbal subgroup $w(G)$ is bounded in terms of $e$ and $w$ only.}
\medskip

This result provides a potentially useful tool for reduction of questions on finite groups to those on nilpotent groups. Historically, tools of this nature played important role in solutions of various problems in group theory, most notably the restricted Burnside problem \cite{hahi}.

Multilinear commutators (outer commutator words) are words which are obtained by nesting commutators, but using always different indeterminates. Thus the word $[[x_1,x_2],[x_3,x_4,x_5],x_6]$ is a multilinear commutator while the Engel word
$[x_1,x_2,x_2,x_2]$ is not.
An important family of multilinear commutators are the simple commutators $\gamma_k$, given by $$\gamma_1=x_1, \qquad \gamma_k=[\gamma_{k-1},x_k]=[x_1,\ldots,x_k].$$

The corresponding verbal subgroups $\gamma_k(G)$ are the terms of the lower central series of $G$. Another distinguished sequence of outer commutator words are the derived words $\delta_k$, on $2^k$ indeterminates, which are defined recursively by
$$
\delta_0=x_1,
\qquad
\delta_k=[\delta_{k-1}(x_1,\ldots,x_{2^{k-1}}),\delta_{k-1}(x_{2^{k-1}+1},\ldots,x_{2^k})].$$
Then $\delta_k(G)=G^{(k)}$, the $k$th derived subgroup of $G$. The values of the word $\delta_k$ sometimes will be called $\delta_k$-commutators.

Let $G$ be a finite group and $P$ a Sylow $p$-subgroup of $G$. An immediate corollary of the Focal Subgroup Theorem \cite[Theorem 7.3.4]{go} is that $P\cap G'$ is generated by commutators. From this one immediately deduces that if all nilpotent subgroups generated by commutators have exponent dividing $e$, then the exponent of the derived group $G'$ divides $e$, too. Thus, the case of Theorem A where $w=[x,y]$ is pretty easy. The proof of the general case uses a number of sophisticated tools, though. In particular it uses the classification of finite simple groups and Zelmanov's solution of the restricted Burnside problem \cite{ze1,ze2}.

\section{Exponent of $G^{(k)}$ in the case of soluble groups}

Throughout the paper we use the expression ``$\{a,b,c,\ldots\}$-bounded" to mean ``bounded from above by some function depending only on $a,b,c, \ldots$". If $A$ is a group of automorphisms of a group $G$, we denote by  $[G,A]$ the subgroup generated by all elements of the form $x^{-1}x^{a}$, where $x\in G,a\in A$. It is well-known that $[G,A]$ is a normal subgroup of $G$. For the proof of the next lemma see for example \cite[6.2.2, 6.2.4]{go}).

\begin{lemma}\label{1}
Let $A$ be a group of automorphisms of a finite group $G$ with
$(|A|,|G|) = 1$.
\begin{enumerate}
\item If $N$ is an $A$-invariant normal subgroup of $G$, then $C_{G/N}(A)=C_{G}(A)N/N$.
\item $G = [G,A]C_{G}(A)$.
\item $[G,A] = [G,A,A]$.
\end{enumerate}
\end{lemma}

We call a subset $B$ of a group $A$ normal if $B$ is a union of conjugacy classes of $A$.
\begin{lemma}\label{2}
Let $A$ be a group of automorphisms of a finite group $G$ with $(|A|,|G|) = 1$. Suppose that $B$ is a normal subset of $A$ such that $A=\langle B\rangle$. Let $k\geq 1$ be an integer. 
Then $[G,A]$ is generated by the subgroups of the form $[G,b_1,\dots,b_k]$, where $b_1,\dots,b_k\in B$.
\end{lemma}
\begin{proof} In view of Lemma \ref{1} it can be assumed that $G=[G,A]$. According to \cite[Lemma 2.1]{golds} the subgroup $[G,A]$ is generated by the subgroups of the form $[P,A]$, where $P$ ranges over the $A$-invariant Sylow subgroups of $G$. Therefore without loss of generality we can assume that $G$ is a $p$-group. It is sufficient to prove the lemma for $G/\Phi(G)$ in place of $G$ so we may further  assume that $G$ is an elementary abelian $p$-group.
Let $H$ be the subgroup of $G$ generated by all the subgroups of the form $[G,b_1,\dots,b_k]$, where $b_1,\dots,b_k\in B$. We need to show that $G=H$. It is clear that $H$ is $A$-invariant.  Passing to the quotient $G/H$ we assume that $H=1$. Thus, $[G,b_1,\dots,b_k]=1$ for all $b_1,\dots,b_k\in B$ and, since $A=\langle B\rangle$, it follows that $[G,b_1,\dots,b_{k-1}]\leq C_G(A)$. Since $G$ is abelian and $G=[G,A]$, we deduce that $C_G(A)=1$ and so $[G,b_1,\dots,b_{k-1}]=1$. Now simply repeating the argument we obtain that also $[G,b_1,\dots,b_{k-2}]=1$ etc. Eventually we see that $[G,b_1]=1$ and hence $[G,A] =1$, as required.\qed
\end{proof}

  Let us call a subgroup $H$ of $G$ a tower of height $h$ if $H$ can be written as a product $H=P_1\cdots P_h$, where

(1) $P_i$ is a $p_i$-group ($p_i$ a prime) for $i=1,\dots,h$.

(2) $P_i$ normalizes $P_j$ for $i<j$.

(3) $[P_i,P_{i-1}]=P_i$ for $i=2,\dots,h$.

It follows from (3) that $p_i\neq p_{i+1}$ for $i=1,\dots,h-1$. A finite soluble group $G$ has Fitting height at least $h$ if and only if $G$ possesses a tower of height $h$ (see for example Section 1 in \cite{turull}).

We will need the following lemmas.

\begin{lemma}\label{1113} Let $G$ be a group and $y$ an element of $G$. Suppose $x_1,\dots,x_{k+1}$ are $\delta_k$-commutators in $G$ for  some $k\geq 0$. Then $[y,x_1,\dots,x_{k+1}]$ is a $\delta_{k+1}$-commutator.
\end{lemma}
\begin{proof} Note that $x_1,\dots,x_k,x_{k+1}$ can be viewed as $\delta_{i}$-commutators for each $i\leq k$. It is clear that $[y,x_1]$ is a $\delta_{1}$-commutator. Arguing by induction on $k$, assume that $k\geq 1$ and $[y,x_1,\dots,x_{k}]$ is a $\delta_{k}$-commutator. Then $ [y,x_1,\dots,x_{k},x_{k+1}] = [[y,x_1,\dots,x_k],x_{k+1}]$ is a $\delta_{k+1}$-commutator.\qed
\end{proof}

\begin{lemma}\label{12} Let $P_1\cdots P_h$ be a tower of height $h$. For every $1\leq k\leq h$ the subgroup $P_k$ is generated by $\delta_{k-1}$-commutators contained in $P_k$.
\end{lemma}
\begin{proof} If $k=1$ the lemma is obvious so we suppose that $k\geq 1$ and use induction on $k$. Thus, we assume that $P_{k-1}$ is generated by $\delta_{k-2}$-commutators contained in $P_{k-1}$. Denote the set of $\delta_{k-2}$-commutators contained in $P_{k-1}$ by $B$. Combining Lemma \ref{2} with the fact that $P_k=[P_k,P_{k-1}]$, we deduce that $P_k$ is generated by subgroups of the form $[P_k,b_1,\dots,b_{k-1}]$, where $b_1,\dots,b_{k-1}\in B$. The result is now immediate from Lemma \ref{1113}.\qed
\end{proof}

\begin{lemma}\label{13} Let $G$ be a group in which every $\delta_k$-commutator has order dividing $e$. Let $H$ be a subgroup of $G$ generated by a set of $\delta_k$-commutators and suppose that for some $j$ the derived subgroup $H^{(j)}$ has exponent $n$. Then $H$ has $\{e,j,n\}$-bounded exponent.
\end{lemma}
\begin{proof} Since a commutator of two $\delta_k$-commutators is again a $\delta_k$-commutator, it follows that all terms of the derived series of $H$ are generated by $\delta_k$-commutators. Therefore every quotient $H^{(i)}/H^{(i+1)}$ has exponent dividing $e$. We deduce now that $H/H^{(j)}$ has exponent dividing $e^j$ and so $H$ has exponent dividing $e^jn$.\qed
\end{proof}

Let us use the symbol $X_k(G)$ to denote the set of all $\delta_k$-commutators in a group $G$. We will now work under the following hypothesis.

\begin{hypothesis}\label{gaga}
Let $e$ and $k$ be positive integers. Assume that $G$ is a finite group such that $x^e=1$ for all $x\in X_k(G)$ and $P^{(k)}$ has exponent dividing $e$ for every $p\in\pi(G)$ and every Sylow $p$-subgroup $P$ of $G$.
\end{hypothesis}
We remark that, in view of Lemma \ref{13}, Hypothesis \ref{gaga} is equivalent to the existence of an $\{e,k\}$-bounded number $e_0$ such that the exponent of every nilpotent subgroup of $G$ generated by a subset of $X_k(G)$ divides $e_0$. It is convenient to work under Hypothesis \ref{gaga} as this condition is obviously inherited by quotients of $G$.

In what follows we will denote by  $F(G)$ the Fitting subgroup of a group $G$ and by $h(G)$ the Fitting height of $G$.

\begin{lemma}\label{height} Assume Hypothesis \ref{gaga}. If $G$ is soluble, then $h(G)$ is $\{e,k\}$-bounded.
 \end{lemma}
\begin{proof} Assume that $G$ is soluble and let $h=h(G)$. Choose a tower $P_1\cdots P_h$ of height $h$ in $G$.  Without loss of generality we assume that $h\geq k$. Lemma \ref{12} tells us that for every $i\geq k+1$ the subgroup $P_i$ is generated by $\delta_k$-commutators contained in $P_i$ (we use here that whenever $i\geq k$ every $\delta_i$-commutator is also a $\delta_k$-commutator). Let $P$ be a Sylow $p$-subgroup of $P_{k+1}\cdots P_h$. By Lemma \ref{13} we conclude that the exponent of $P$ is $\{e,k\}$-bounded. This is true for every Sylow $p$-subgroup of $P_{k+1}\cdots P_h$. Thus, the exponent of $P_{k+1}\cdots P_h$ is $\{e,k\}$-bounded. According to the Hall-Higman theory  \cite{hahi} the Fitting height of a finite soluble group of exponent $n$ is $n$-bounded so we deduce that the Fitting height of $P_{k+1}\cdots P_h$ is $\{e,k\}$-bounded. It is clear that the Fitting height of $P_1\cdots P_{k}$ is at most $k$. Therefore $h$ (the Fitting height of $P_1\cdots P_h$) is $\{e,k\}$-bounded. The lemma follows.\qed
\end{proof}

Let $G$ be a group and $w=w(x_{1},\ldots,x_{n})$ a word. The marginal subgroup $w^{*}(G)$ of $G$ corresponding to the word $w$ is defined as the set of all $a\in G$ such that  
$$w(g_{1},\ldots,ag_{i},\ldots,g_{n})=w(g_{1},\ldots,g_{i}a,\ldots,g_{n})=w(g_{1},\ldots,g_{i},\ldots,g_{n}),$$ for all $g_{1},\ldots,g_{n} \in G$ and $ 1\leq i\leq n$. It is well known that $w^{*}(G)$ is a characteristic subgroup of $G$ and that $[w^{*}(G),w(G)]=1.$
If $w$ is a multilinear commutator, then $w^{*}(G)$ is precisely the set $S$ such that  $w(g_{1},\ldots, g_{n})=1$ whenever at least one of the elements $g_{1},\ldots,g_{n}$ belongs to $S$. A proof of this can be found in \cite[Theorem 2.3]{turner}.
The following helpful lemma was communicated to me by C. Acciarri and G. A. Fern\'andez-Alcober.
\begin{lemma}\label{easy}  Let $w$ be a multilinear commutator and $G$ a group with a normal subgroup $N$ that contains no nontrivial $w$-values. Then $[N,w(G)]=1$.
\end{lemma}
\begin{proof} Let $w=w(x_1,\dots,x_k)$. Since $N$ is normal in $G$ and $w$ is a multilinear commutator, it follows that $w(g_1,\dots,g_k)$  belongs to $N\cap G_{w}$ whenever at least one of the elements $g_1,\dots,g_k$ belongs to $N$. Thus by the hypothesis all such elements must be trivial. It follows that $N\subseteq w^{*}(G)$, where $w^*(G)$ is the marginal subgroup of $G$ corresponding to $w$. The result is now clear since $w^{*}(G)$ always commutes with $w(G)$.$\quad$ \qed
\end{proof}

\begin{lemma}\label{imp} Assume Hypothesis \ref{gaga}. If $G$ is soluble, then the exponent of $G^{(k)}$ is $(e,k)$-bounded.
\end{lemma}
\begin{proof} Assume that $G$ is soluble and let $h=h(G)$. By Lemma \ref{height},  $h$ is $\{e,k\}$-bounded. If $G$ is nilpotent, the result is immediate so we assume that $h\geq 2$ and use induction on $h$. Let $H=G^{(k)}$ and $N=F(G)$. By induction the exponent of $HN/N$ is $\{e,k\}$-bounded. Let $M$ be the subgroup generated by all $\delta_k$-commutators of $G$ contained in $N$. By the remark made after Hypothesis \ref{gaga}, the exponent of $M$ is $\{e,k\}$-bounded, too. Let us pass to the quotient $G/M$ and assume that $M=1$. Lemma \ref{easy} tells us that in this case $[N,H]=1$. Hence, the exponent of $H/Z(H)$ is $\{e,k\}$-bounded. Mann showed that in any finite group $K$ the exponent of the derived group $K'$ is bounded in terms of the exponent of $K/Z(K)$ \cite{mann}. Using Mann's result, we deduce that $H'$ has $\{e,k\}$-bounded exponent. Therefore we can pass to the quotient $G/H'$ and without loss of generality assume that $H$ is abelian. Since $H$ is generated by elements of order dividing $e$, it follows that the exponent of $H$ divides $e$, as required.\qed
\end{proof}

Combining Lemma \ref{imp} with Lemma \ref{13} we deduce the following.
\begin{lemma}\label{clar} Assume Hypothesis \ref{gaga}. If $T$ is a soluble subgroup of $G$ such that $T$ is generated by a set of $\delta_k$-commutators, then the exponent of $T$ is $\{e,k\}$-bounded.
\end{lemma}

\section{Exponent of $G^{(k)}$ in the case of arbitrary groups}

Let $G$ be a finite group and $k$ a positive integer. As in  \cite{deltacom} we will associate with $G$ a triple of numerical parameters $n_k(G)=(\lambda,\mu,\nu)$ where
the parameters $\lambda,\mu,\nu$ are defined as follows.

If $G$ is of odd order, we set $\lambda=\mu=\nu=0$. Suppose that $G$ is
of even order and choose a Sylow 2-subgroup $P$ in $G$. 
If the derived length $dl(P)$ of $P$ is at most $k+1$, we define
$\lambda=dl(P)-1$. Put $\mu=2$ if $X_{\lambda}(P)$
contains elements of order greater than two and $\mu=1$ otherwise. We
let $\nu=\mu$ if $X_{\lambda}(P)\not\subseteq Z(P)$ and $\nu=0$ if
$X_{\lambda}(P)\subseteq Z(P)$.

If the derived length of $P$ is at least $k+2$, we define $\lambda=k$. 
Then $\mu$ will denote the number with the property that $2^\mu$ is the
maximum of orders of elements in $X_k(P)$. Finally, we let $2^{\nu}$ be
the maximum of orders of commutators $[a,b]$, where $b\in P$ and $a$ is an
involution in a cyclic subgroup generated by some element from $X_k(P)$.
 
The set of all possible triples $n_k(G)$ is naturally endowed with the lexicographical order. Following the terminology used by Hall and Higman \cite{hahi} we call a group $G$ monolithic if it has a unique minimal normal subgroup which is non-abelian simple. In the modern literature such groups are very often called ``almost simple". For the proof of the next proposition see \cite{deltacom}.

\begin{proposition}\label{4} Let $k\geq 1$ and $G$ a group of
even order such that $G$ has no nontrivial normal soluble
subgroups. Then $G$ possesses a normal subgroup $L$ such that $L$ is
residually monolithic and $n_k(G/L)<n_k(G)$.
\end{proposition}

The next lemma is taken from \cite{vari}. The proof is based on Lie-theoretic techniques created by Zelmanov.
The lemma plays a crucial role in the proof of Lemma \ref{55} which in turn is fundamentally important for the proof of Theorem A.
\begin{lemma}\label{22222} Let $G$ be a group in which every
${\delta_k}$-commutator is of order dividing $e$. Let $H$
be a nilpotent subgroup of $G$ generated by a set of
$\delta_k$-commutators. Assume that $H$ is in fact $m$-generated
for some $m\geq 1$. Then the order of $H$ is $\{e,k,m\}$-bounded.
\end{lemma}

The lemma that follows partially explains why Proposition \ref{4} is
important for the proof of Theorem A. The proof can be found in \cite{deltacom}.

\begin{lemma}\label{55} There exist  $\{e,k\}$-bounded numbers
$\lambda_0,\mu_0,\nu_0$ with the property that if $G$ is a group
in which every ${\delta_k}$-commutator is of order dividing $e$, then
$n_k(G)\leq(\lambda_0,\mu_0,\nu_0)$.
\end{lemma}

\begin{proposition}\label{789} Under Hypothesis \ref{gaga} the exponent of $G^{(k)}$ is $\{e,k\}$-bounded.
\end{proposition}

\begin{proof} According to Lemma \ref{55} the number of all
triples that can occur as $n_k(G)$ is $\{e,k\}$-bounded. We therefore
can use induction on $n_k(G)$. If $n_k(G)=(0,0,0)$, then $G$ has odd order.
By the Feit-Thompson Theorem \cite{fetho} $G$ is soluble, so the conclusion holds by Lemma \ref{imp}. Hence, we assume that $n_k(G)>(0,0,0)$ and there exists an $\{e,k\}$-bounded number $E_0$ with the property that if $L$ is a normal subgroup such that $n_k(G/L)<n_k(G)$, then the exponent of $G^{(k)}L/L$ is at most $E_0$.

Suppose first that $G$ has no nontrivial normal soluble
subgroups. Proposition \ref{4} tells us that $G$ possesses 
a normal subgroup $L$ such that $L$ is residually monolithic and
$n_k(G/L)<n_k(G)$. A result of Jones \cite{Jones} says that any infinite family of finite simple 
groups generates the variety of all groups.  Observe that every monolithic group is isomorphic to a subgroup of  $\mathrm{Aut}(H)$, where $H$ is the non-abelian simple group isomorphic to the unique minimal normal subgroup of the monolithic group.  Thus by Jones' result  we have only finitely many possibilities to choose the group $H$  in which every ${\delta_k}$-commutator is of order dividing $e$. It follows that up to isomorphism there exist only finitely many monolithic groups in which every ${\delta_k}$-commutator is of order dividing $e$.  The exponent of such monolithic groups is $\{e,k\}$-bounded and so $L$ is residually of $\{e,k\}$-bounded exponent. Therefore $L$ has $\{e,k\}$-bounded exponent.  We conclude that the exponent of $G^{(k)}$ is $\{e,k\}$-bounded.

Now let us drop the assumption that $G$ has no nontrivial normal soluble subgroups. Let $S$ be the product of all normal soluble subgroups of $G$. The above paragraph shows that $G^{(k)}S/S$ has $\{e,k\}$-bounded exponent. Let $T$ be the subgroup generated by all $\delta_k$-commutators contained in $S$. By Lemma \ref{imp} the exponent of $T$ is $\{e,k\}$-bounded. Passing to the quotient $G/T$ we can assume that $T=1$. Combining our hypothesis with Lemma \ref{easy} we conclude that $S\cap G^{(k)}\leq Z(G^{(k)})$. Hence, the exponent of $G^{(k)}/Z(G^{(k)})$ is $\{e,k\}$-bounded. We now use Mann's theorem \cite{mann} and deduce that $G^{(k+1)}$ has $\{e,k\}$-bounded exponent. Therefore we can pass to the quotient $G/G^{(k+1)}$ and without loss of generality assume that $G^{(k)}$ is abelian. Since $G^{(k)}$ is generated by elements of order dividing $e$, it follows that the exponent of $G^{(k)}$ is $\{e,k\}$-bounded, as required.\qed
\end{proof}

\section{Proof of Theorem A}

The number of distinct indeterminates used in the expression of a multilinear commutator $w$ is the weight of $w$. The following lemma is taken from \cite{lola}.

\begin{lemma}\label{um} Let $G$ be a group and $w$ a multilinear
commutator of weight $k$. Every $\delta_k$-commutator in $G$ is
a $w$-value.
\end{lemma}
\begin{lemma}\label{dois} Let $w$ be a multilinear commutator of weight $k$,  and let $G$ be a soluble group of derived length at most $k$. Suppose that all $w$-values in $G$ have order dividing $e$. Then the exponent of the verbal subgroup $w(G)$ is $\{e,k\}$-bounded.
\end{lemma}
\begin{proof} Let $T$ be the last nontrivial term of the derived series of $G$. Since $T$ is abelian, the subgroup generated by all $w$-values contained in $T$ has exponent dividing $e$. Passing to the quotient by this subgroup, we can assume that no $w$-value lies in $T\setminus\{1\}$. Now Lemma \ref{easy} shows that $[T,w(G)]=1$. The induction on the derived length of $G$ tells us that $w(G)/Z(w(G))$, the image of $w(G)$ in $G/C_G(w(G))$, has bounded exponent. Therefore Mann's theorem \cite{mann} implies that the derived group $w(G)'$ has $\{e,k\}$-bounded exponent. We can now pass to the quotient $G/w(G)'$ and assume that $w(G)$ is abelian. But now it is clear that since all $w$-values in $G$ have order dividing $e$, the exponent of $w(G)$ must divide $e$, too.\qed
\end{proof}
Theorem A is now immediate. 
\begin{proof} Suppose that $w$ is a multilinear commutator and $G$ is a finite group in which any nilpotent subgroup generated by $w$-values has exponent dividing $e$. By Lemma \ref{um} there exists $k\geq 1$ such that every $\delta_k$-commutator is a $w$-value. This $k$ depends only on $w$. Obviously $G$ satisfies Hypothesis \ref{gaga}. Proposition \ref{789} now tells us that the exponent of $G^{(k)}$ is $\{e,k\}$-bounded. It is straightforward from Lemma \ref{dois} that the exponent of $w(G)/G^{(k)}$ is likewise $\{e,k\}$-bounded. The proof is complete.\qed
\end{proof}



\end{document}